\documentclass[12pt,reqno]{amsart}
\usepackage{amsmath}
\usepackage{amssymb}
\usepackage{amsthm}
\usepackage{enumerate}
\usepackage[mathscr]{eucal}
\usepackage{cite}
\usepackage[dvips]{graphics}
\usepackage{color,soul}
\usepackage{hyperref}

\newtheorem{theorem}{Theorem}[section]
\newtheorem{lemma}{Lemma}[section]
\newtheorem{proposition}{Proposition}[section]
\newtheorem{corollary}{Corollary}[section]
\theoremstyle{definition}
\newtheorem{definition}{Definition}[section]
\newtheorem{example}{Example}[section]

\theoremstyle{remark}
\newtheorem{remark}{Remark}[section]
\numberwithin{equation}{section}

\begin{document}
\title[Using an implicit function to prove common fixed point theorems]
{Using an implicit function to prove common fixed point theorems}
\author[Imdad, Gubran and Ahmadullah]{Mohammad Imdad$^{1}$, Rqeeb Gubran$^{2}$ and Md Ahmadullah$^{3}$ }

\maketitle

\begin{center}
{\footnotesize Department of Mathematics, Aligarh Muslim University, Aligarh-202002, India.\\
 mhimdad@gmail.com$^{1}$,
rqeeeb@gmail.com$^{2}$, ahmadullah2201@gmail.com$^{3}$\\}

\end{center}
\begin{abstract}
	In this paper, we prove common fixed point results for a self-mappings satisfying an implicit function which is general enough to cover a multitude of known as well as unknown contractions. Our results modify, unify, extend and generalize many relevant results of the existing literature. Interestingly, unlike several other cases, our main results deduce a nonlinear order-theoretic version of a well-known fixed point theorem (proved for quasi-contraction) due to \'{C}iri\'{c} (Proc. Amer. Math. Soc. (54) 267-273, 1974). Finally, in the setting of metric spaces, we drive a sharpened version of Theorem 1 due to Berinde and Vetro (Fixed Point Theory Appl. 2012:105).
\end{abstract}

\vspace{.3cm} \noindent {\bf Keywords}: Implicit function, common fixed point, quasi-contraction, comparison function.

\vspace{.3cm}\noindent {\bf 2010 AMS Classification}: 47H10; 54H25

\section{Introduction and Preliminaries}

 Prior to Popa \cite{popa1997fixed, popa1999some}, researchers of metric fixed point theory use to prove a theorem for every contraction condition which amounts to saying that once there is a relatively new contraction, one is required to prove a separate theorem for the same. But Popa \cite{popa1997fixed} initiated the idea of implicit function with a view to cover several contraction conditions in one go.\\

 In recent years, the idea of implicit function has been utilized by several authors and by now, there exists a considerable literature on this theme. To mention a few, one can be referred to \cite{ali2009unifying,altun2009some,aliouche2007common,popa2005general,aliouche2009general,berinde2012approximating,popa2001general,popa2002fixed,vetro2013common,
 ullah2016unified,popa2010using,imdad2008general,ali2008implicit,imdad2002remarks,imdad2013employing} and references therein. Also, one of the interesting articles on this theme is due to Berinde and Vetro \cite{berinde2012common} wherein authors proved results on coincidence as well as common fixed point  for a general class of self-mappings covered under an implicit function in the settings of metric and ordered metric spaces.  However, we observe that the order-theoretical result of \cite{berinde2012common} is not correct in its present form.\\

Before undertaking specific discussions, we recall  the background material needed in our subsequent discussions.\\

 We denote by $\mathbb{R}$, $\mathbb{N}$ and $\mathbb{N}_0$ respectively the set of all real numbers, the set of natural numbers  and the set $\mathbb{N}\cup\{0\} $. As usual, $I_X$ denotes the identity mapping defined on $X$. For brevity, we write $Tx$ instead of $T(x)$.

\begin{definition} \cite{o2008fixed} A triplet
	$(X,d,\preceq)$ is called an ordered metric space if $(X,d)$ is a
	metric space and $(X,\preceq)$ is an ordered set. Moreover, two elements $x,y\in X$ are said to be comparable if either $x\preceq y$ or $x\succeq y$. For brevity,
	we denote it by $x\prec\succ y$.

\end{definition}

Let $\{x_n\}$ be a sequence in an ordered metric space $(X,d,\preceq)$. Then, if $\{x_n\}$ is an increasing (resp. decreasing, monotone) and converges to $x$, we denote it by $x_n\uparrow x$ (resp.  $x_n\downarrow x$,  $x_n\uparrow\downarrow x$).

\begin{definition} \cite{ciric2009monotone}\label{1.5} Let $(T,S)$ be a pair of self-mappings on an ordered metric space $(X,d,\preceq)$. Then the mapping $T$ is said to be
\begin{enumerate}
\item[{($i$)}] $S$-increasing if (for any $x,y\in X$)
$Sx\preceq Sy\Rightarrow Tx\preceq Ty$,
\item[{($ii$)}] $S$-decreasing if (for any $x,y\in X$)
$Sx\preceq Sy\Rightarrow Tx\succeq Ty$,
\item[{($iii$)}] $S$-monotone if $T$ is either $S$-increasing or
$S$-decreasing.
\end{enumerate}
\end{definition}

On setting $S=I_X$,  Definition \ref{1.5}$(i)$ (resp. \ref{1.5}$(ii)$, \ref{1.5}$(iii)$) is reduced to the usual definition of the increasing (resp. decreasing, monotone) self-mapping ($T$ on $X$).

\begin{definition}\cite{vetro2013common} Let $(T,S)$ be a pair of self-mappings on an ordered metric space $(X,d,\preceq)$ with $T(X)\subseteq S(X)$. For every $x_0\in X$, consider the sequence $\{x_n\}\subset X$ defined by $Tx_n=Sx_{n+1}$, for all $n\in \mathbb{N}_0$. Then $\{Tx_n\}$ is called $T$-$S$-sequence with initial point $x_0$.
\end{definition}

\begin{definition}\cite{jungck1976commuting} Let $(T,S)$ be a pair of self-mappings on an ordered metric space $(X,d,\preceq)$. Then
a coincidence point of the pair $(T,S)$ is a point $x\in X$ such that $Sx=Tx$. If  $x^*\in X$ is such that  $Sx=Tx=x^*$, then   $x^*$ is called a point of coincidence of the pair $(T, S).$ If  $x^*=x$, then $x$ is said to be a common fixed point.

\end{definition}

\noindent By $C(T,S)$, We denote the set of all coincidence points of the pair $(T,S)$.

\begin{definition} \cite{jungck1986compatible,jungck1996common,alam2015enriching} Let $(T,S)$ be a pair of self-mappings on an ordered metric space $(X,d,\preceq)$. Then
 the pair is said to be
\begin{enumerate}[(i)]
\item compatible if $\lim\limits_{n\to \infty}d(S(Tx_n),T(Sx_n))=0,$ whenever $\{x_n\}$ is a sequence in $X$ such that
$\lim\limits_{n\to \infty}Sx_n=\lim\limits_{n\to \infty}Tx_n,$

\item  $\overline{\rm O}$-compatible if for any sequence $\{x_n\}\subset X$ with $Sx_n\uparrow z\;{\rm and}\;Tx_n\uparrow z$ (for some $z\in X$) implies $\lim\limits_{n\to \infty}d(S(Tx_n),T(Sx_n))=0,$

\item $\underline{\rm O}$-compatible if for any sequence $\{x_n\}\subset X$  $\{x_n\}\subset X$ with $Sx_n\downarrow z\;{\rm and}\;Tx_n\downarrow z$ (for some $z\in X$) implies $\lim\limits_{n\to \infty}d(S(Tx_n),T(Sx_n))=0,$

\item {\rm O}-compatible if for any sequence $\{x_n\}\subset X$ with
$Sx_n\uparrow\downarrow z\;{\rm and}\;Tx_n\uparrow\downarrow z$ (for some $z\in X$) implies $\lim\limits_{n\to \infty}d(S(Tx_n),T(Sx_n))=0.$

\item  weakly compatible if $S(Tx)=T(Sx)$, for every coincidence point $x\in X$.

\end{enumerate}
\end{definition}

\begin{remark} \label{rem5} In an ordered metric space, compatibility $\Rightarrow${O}-compatibility $\Rightarrow \overline{\rm O}$-compatibility (as well as $\underline{\rm O}$-compatibility) $\Rightarrow$ weak compatibility. \end{remark}

\begin{definition} \cite{alam2015enriching}  An ordered metric space $(X,d,\preceq)$ is called  $\overline{\rm O}$-complete (resp. $\underline{\rm O}$-complete,  {\rm O}-complete) if every increasing (resp. decreasing, monotone) Cauchy sequence converges in $X$.\end{definition}


\begin{remark} \label{rem1} In an ordered metric space, completeness $\Rightarrow$ {\rm O}-completeness $\Rightarrow \overline{\rm O}$-completeness (as well as $\underline{\rm O}$-completeness). \end{remark}

\begin{definition} \cite{alam2015comparable} Let $(T,S)$ be a pair of self-mappings on an ordered metric space $(X,d,\preceq)$. We say that
$(X,\preceq)$ is $(T,S)$-directed if for every pair $x,y\in X$,
$\exists~z\in X$ such that $Tx\prec\succ Sz$ and $Ty\prec\succ
Sz$.\end{definition}

Particularly, for $S=I_X$, $(X,\preceq)$ is called $T$-directed.

\begin{definition} \cite{alam2015comparable} Let $T$ be self-mapping on an ordered metric space $(X,d,\preceq)$. We say that $T$ is comparable mapping if it maps comparable elements to comparable elements.\end{definition}

\begin{definition}\label{g-cnts} \cite{sastry2000common} Let $(T,S)$ be a pair of
self-mappings on a metric space $(X,d)$ and $x\in X$. We say that
$T$ is $S$-continuous at $x$ if for any sequence $\{x_n\}\subset X$,
$$Sx_n\stackrel{d}{\longrightarrow} Sx\Rightarrow Tx_n\stackrel{d}{\longrightarrow} Tx.$$
Moreover, $T$ is called $S$-continuous if it is $S$-continuous at
every point of $X$.
\end{definition}

Notice that, on setting $S=I_X,$ Definition \ref{g-cnts} reduces to the usual definition of continuity.

\begin{definition} \label{go-cnts}\cite{alam2015enriching} Let $(T,S)$ be a pair of
self-mappings on a metric space $(X,d)$ and $x\in X$.
Then $T$ is called $(S,\overline{\rm O})$-continuous (resp. $(S,\underline{\rm O})$-continuous, $(S,{\rm O})$-continuous) at $x$ if $Tx_n\stackrel{d}{\longrightarrow} Tx,$ for every sequence $\{x_n\}\subset X$ with $Sx_n\uparrow Sx$ (resp.  $Sx_n\downarrow Sx$, $Sx_n \uparrow\downarrow Sx$). Moreover, $T$ is called $(S,{\rm O})$-continuous (resp.
$(S,\overline{\rm O})$-continuous, $(S,\underline{\rm O})$-continuous) if it is $(S,{\rm O})$-continuous (resp.
$(S,\overline{\rm O})$-continuous, $(S,\underline{\rm
O})$-continuous) at every point of $X$.
\end{definition}

On setting $S=I_X$,  Definition \ref{go-cnts}  is reduces to the usual definition of the $\overline{\rm O}$-continuity (resp. $\underline{\rm O}$-continuity, ${\rm O}$-continuity) of a self-mapping ($T$ on $X$).


\begin{remark} \label{gocontinuity}In an ordered metric space, $S$-continuity $\Rightarrow$ $(S,{\rm O})$-continuity $\Rightarrow$ $(S,\overline{\rm O})$-continuity (as well as $(S,\underline{\rm O})$-continuity).
 \end{remark}

\begin{lemma}\label{prop1} \cite{alam2014some} Let $(T, S)$ be a pair of self-mappings defined on an ordered set $(X,\preceq)$. If $T$ is
$S$-monotone and
$Sx=Sy$, then $Tx=Ty$.
\end{lemma}

\begin{lemma}\label{prop2} \cite{alam2014some} If the pair $(T,S)$ is weakly
compatible, then every point of coincidence of the pair remains a
coincidence point of it.
 \end{lemma}

\begin{lemma}\label{lem2}\cite{haghi2011some} Let $T$ be a self-mapping on a non-empty set $X$. Then, there exists a subset $A\subseteq X$
such that $T(A)=T(X)$ and $T: A \rightarrow X$ is one-one.
\end{lemma}

\section{Implicit function}
In order to describe our implicit function, let us define comparison function:
\begin{definition} \cite{matkowski1975integrable}\label{defcompa} A function $\varphi:[0,\infty)\to [0,\infty)$ is called comparison function (or Matkowski type function) if it satisfies the following conditions:
	\begin{enumerate}[(i)]
		\item $\varphi$ is increasing on $[0, \infty)$,
		\item $\lim\limits_{n\to\infty}\varphi^n(t)=0,\text{ for all }t>0.$
	\end{enumerate}
\end{definition}

Now, we record some results which needed later to prove our results. We begin by the following lemma which highlights some basic proprieties of the comparison function.
\begin{lemma}\label{samtur} Let $\varphi$ be a comparison function, then
	\begin{enumerate}[(i)]
		\item $\varphi(t)<t$, for all $t>0$.
		\item $\varphi(0)=0$.
	\end{enumerate}
\end{lemma}

\begin{proof}~ 	$(i)$ Assume there exists $t_0>0$ with $t_0\leq \varphi(t_0)$. Since $\varphi$ is increasing
	$\varphi(t_0)\leq \varphi^2(t_0)$, it follows that $t_0\leq\varphi(t_0)\leq \varphi^2(t_0)$.
	 In general $t_0\leq \varphi^n(t_0)$ for all $n\in \{1,2,...\}$, and by letting
	 $n \rightarrow \infty$ we get  $t_0\leq 0$ which contradicts our assumption.
	
	 \indent $(ii)$ On the contrary, suppose $\varphi (0)=t$ for some $0<t$. Since $0<t$ and
	  $\varphi$ is increasing, $\varphi(0) < \varphi(t)$. It follows that $t<\varphi(t)$, which contradicts (i).
\end{proof}

Now, we consider the family $\mathfrak{F}$ of all real continuous functions $F:\mathbb{R}^6_+\rightarrow\mathbb{R}_+$. In the respect of the family $\mathfrak{F}$, the following conditions will be utilized in our results:

\begin{itemize}
	\item[($F_{1a}$)] $F$ is decreasing in the fifth variable and there exist a comparison function $\varphi$ such that $F(u,v,v,u,u+v,0)\leq 0$, for $u,v\geq 0$ implies that $u\leq \varphi(v).$
	\item[($F_{1b}$)] $F$ is decreasing in the fourth variable and there exist a comparison function $\varphi$ such that $F(u,v,0,u+v,u,v)\leq 0$, for $u,v\geq 0$ implies that  $u\leq \varphi(v).$
	\item[($F_{1c}$)] $F$ is decreasing in the third variable and  there exist a comparison function $\varphi$ such that $F(u,v,u+v,0,v,u)\leq 0$, for $u,v\geq 0$ implies that  $u\leq\varphi(v).$
	\item[($F_2$)] $F(u,u,0,0,u,u)>0$, for all $u>0$.
\end{itemize}

\vspace{.3cm} In \cite{berinde2012approximating}, Berinde considered the family  $\mathcal{F}$ of all real continuous functions $F:\mathbb{R}^6_+\rightarrow\mathbb{R}_+$ and the following conditions:

\begin{itemize}	
	\item[($f_{1a}$)] $F$ is decreasing in the fifth variable and $F(u,v,v,u,u+v,0)\leq 0$, for $u,v\geq 0$ implies that there exists $h\in [0,1)$ such that $u\leq hv.$	
	\item[($f_{1b}$)] $F$ is decreasing in the fourth variable and $F(u,v,0,u+v,u,v)\leq 0$, for $u,v\geq 0$ implies that there exists $h\in [0,1)$ such that $u\leq hv.$	
	\item[($f_{1c}$)] $F$ is decreasing in the third variable and $F(u,v,u+v,0,v,u)\leq 0$, for $u,v\geq 0$ implies that there exists $h\in [0,1)$ such that $u\leq hv.$	
	\item[($f_2$)] $F(u,u,0,0,u,u)>0$, for all $u>0$.\\\end{itemize}

 Observe that $\mathcal{F}\subseteq\mathfrak{F}$.

\begin{definition}
	If $\rho$ is a comparison function such that $\varphi$ defined by: $\varphi(t)=\rho(2t)$, for $t\geq0,$ is a comparison function, then $\rho$ is called  a half-comparison function.
\end{definition}

For example, the comparison function $\rho(t)=kt$ is a half-comparison for $k\in[0,1/2)$ while it is not half-comparison for $k\in[1/2,1)$.

\begin{proposition}\label{prphalf}
	Let $\rho$ be a half-comparison function. Then $\rho$ is a comparison function with $\rho(2t)<t$, for all $t>0.$
\end{proposition}
\begin{proof}
	As $\rho$ is a half-comparison function, there exists a comparison function $\varphi$ such that $\varphi(t)=\rho(2t)$. Thus,
	$\rho(2t)=\varphi(t)<t,\text{ for all }t>0$
\end{proof}

The following functions satisfy variety of the conditions $F_{1a}-F_2$. In all the following examples, $\psi$ is a continuous comparison function while $\rho$ is a continuous half-comparison function.
\begin{example}
All functions $F$ defined in Examples 3.1-3.8, 3.17 and 3.19 of \cite{berinde2012approximating} are in $\mathfrak{F}$ and 
satisfy conditions $F_{1a}-F_2$ for $\varphi(t)=kt$ with a suitable $k$.
\end{example}
\begin{example}\cite[Example 2]{berinde2012common}\label{lin.quasi}
	Consider the function  $F\in\mathfrak{F}$, given by:
	$$F(t_1,t_2,t_3,t_4,t_5,t_6)=t_1-k \big(max\left\{t_2,t_3,t_4,t_5,t_6\right\}\big),k\in[0,1/2),$$
then $F$ satisfies $F_{1a}-F_{2}$ with $\varphi(t)=\frac{kt}{1-k}$.
\end{example}
\begin{example}\label{nonlin.quasi}
	Define  $F\in\mathfrak{F}$ given by:
	$$F(t_1,t_2,t_3,t_4,t_5,t_6)=t_1-\rho\big(max\left\{t_2,t_3,t_4,t_5,t_6\right\}\big).$$
	
\vspace{.3cm}
Let $u,v\geq0$ and choose a comparison function $\varphi$ where $\varphi(t)=\rho(2t),\forall~t\in[0,\infty).$ Assume $u\geq v$ so that
	$u- \rho(u+v)\leq0$ $\Rightarrow u\leq \rho(2u)=\varphi(u)$ which is a contradiction. Thus, $u\leq v$ and $u\leq \rho(u+v)\leq \varphi(v).$ Therefore, $F$ satisfies $F_{1a}$ with $\varphi$ given by $\varphi(t)=\rho(2t),t>0$.
	
	Similarly, we can prove that $F$ satisfies $F_{1b},$ $F_{1c}$ and $F_2$ for the same $\varphi$.
\begin{remark} $F$ defined in Example \ref{lin.quasi} is a special case of $F$ defined in Example \ref{nonlin.quasi}.
	\end{remark}
\end{example}

\begin{example}\label{1.10.of.V.and.V}
	Define  $F\in\mathfrak{F}$ as:
	$$F(t_1,t_2,t_3,t_4,t_5,t_6)=t_1-\psi \bigg(t_3\frac{t_5+t_6}{t_2+t_4}\bigg).$$ Then  $F$ satisfies $F_{1a}$ and $F_2$ with $\varphi=\psi$ but does not satisfy $F_{1b}$ and $F_{1c}$.
\end{example}

\begin{example}\label{1a,1b.not.1c.and.2}
	Define  $F\in\mathfrak{F}$ as:
	$$F(t_1,t_2,t_3,t_4,t_5,t_6)=t_1-\psi\bigg(t_2\frac{t_5+t_6}{t_3+t_4}\bigg).$$ Then  $F$ satisfies $F_{1a}$ and $F_{1c}$ with $\varphi=\psi$, while $F_2$ is not applicable.
\end{example}

\begin{example}
	Define  $F\in\mathfrak{F}$ as:
	$$F(t_1,t_2,t_3,t_4,t_5,t_6)=t_1-\psi(t_2).$$ Then  $F$ satisfies $F_{1a}-F_2$ with $\varphi=\psi$.
\end{example}
\begin{example}\label{not.1c.not2n}
	Define  $F\in\mathfrak{F}$ as:
	$$F(t_1,t_2,t_3,t_4,t_5,t_6)=t_1-\rho(t_3+t_4).$$ Then,  $F$ satisfies $F_{1a}-F_{2}$ with $\varphi$ given by $\varphi(t)=\rho(2t),t>0$.
	\noindent Observe that, if we replace $\rho$ by $\psi$ in this example, then $F$ satisfies condition $F_2$ only.
\end{example}
\begin{example}\label{not.1c.not2}
	Define  $F\in\mathfrak{F}$ as:
	$$F(t_1,t_2,t_3,t_4,t_5,t_6)=t_1-\rho(t_2+t_3).$$ Then,  $F$ satisfies $F_{1a},F_{1b}$ and $F_{2}$ with $\varphi$ given by $\varphi(t)=\rho(2t),t>0$.
\end{example}
	\noindent Observe that, if we replace $\rho$ by $\psi$ in this example, then $F$ satisfies conditions $F_{1b}$ and $F_2$.

\begin{example}
	Define  $F\in\mathfrak{F}$ as:
	$$F(t_1,t_2,t_3,t_4,t_5,t_6)=t_1-\psi\bigg(max\left\{t_2,\frac{t_3+t_4}{2},t_5,t_6\right\}\bigg).$$ Then,  $F$ satisfies $F_{1b}-F_2$ with $\varphi=\psi$.
	
\end{example}

\begin{example}
	Define  $F\in\mathfrak{F}$ as:
	$$F(t_1,t_2,t_3,t_4,t_5,t_6)=t_1-\psi \bigg(max\left\{t_2,t_3,t_4,\frac{t_5+t_6}{2}\right\}\bigg).$$ Then,  $F$ satisfies $F_{1a}$ and $F_2$ with $\varphi=\psi$.
\end{example}
\begin{example}
	Define  $F\in\mathfrak{F}$ as:
	$$F(t_1,t_2,t_3,t_4,t_5,t_6)=t_1-\psi \bigg(max\left\{t_2,t_3,t_4,\frac{t_5+t_6}{2}\right\}\bigg)-L min \{t_3,t_4,t_5,t_6\},L\geq0.$$ Then,  $F$ satisfies $F_{1a}$ and $F_2$ with $\varphi=\psi$.
\end{example}



\begin{example}\label{mine}
	Define  $F\in\mathfrak{F}$ as: $$F(t_1,t_2,t_3,t_4,t_6)=t_1-\psi\bigg(max\left\{t_2,\frac{1}{2}[t_3+t_4],\frac{1}{2}[t_5+t_6]\right\}\bigg).$$ Then
$F$ satisfies $F_{1a}-F_{2}$ with $\varphi=\psi$.\\
\end{example}
\begin{example}
	Define  $F\in\mathfrak{F}$ as:
	$$F(t_1,t_2,t_3,t_4,t_5,t_6)=t_1-\psi\bigg(max\left\{t_2,t_3,\frac{t_4}{2},\frac{t_5+t_6}{2},t_6\right\}\bigg).$$ Then,  $F$ satisfies $F_{1a},F_{1b}$ and $F_2$ with $\varphi=\psi$.\\
\end{example}	
	
	

The following theorem is essentially contained in Berinde and Vetro \cite{berinde2012common}:

\begin{theorem} \label{berinde-vetro}\cite[Theorem 2]{berinde2012common} Let $(X,d,\preceq)$ be a complete ordered metric space and $(T,S)$ a pair of self-mappings on $X$ such that $T(X)\subseteq S(X)$ and $T$ is $S$-increasing. Assume that there exists a function $F \in \mathcal{F}$ satisfying $f_{1a},$ such that for all $x,y\in X$ with $Sx\preceq Sy$,
	\begin{equation}\label{rr}
	F(d(Tx,Ty),d(Sx,Sy),d(Sx,Tx),d(Sy,Ty),d(Sx,Ty),d(Sy,Tx))\leq0.
	\end{equation}
	If following conditions hold:
	\begin{description}
		\item[($a_1$)] there exists $x_{0}\in X$ such that $Sx_0\preceq Tx_0$,
		\item[($a_2$)] for every increasing sequence
		$\{Sx_n\}$ in $X$ converges to $Sx$, we have $Sx_n\preceq Sx,\;\;\forall~ n\in \mathbb{N}_0 \text{ and } Sx\preceq S(Sx)$.
	\end{description}
	then, the pair  $(T,S)$ has a coincidence point in $X$. Moreover, if
	\begin{description}
		\item[($a_3$)] $(T,s)$ is weakly compatible pair,
		\item[($a_4$)] $F$ satisfies $f_2$,
	\end{description}
	then, the pair  $(T,S)$ has a common fixed point. Further for any $x_0\in X$, the $T$-$S$-sequence $\{Tx_n\}$ (with initial point $x_0$) converges to a common fixed point of the pair.
\end{theorem}

\vspace{.3cm}
The authors in \cite{berinde2012common}, also, gave the following sufficient conditions for the uniqueness of the common fixed point in above theorem:
\begin{description}
	\item[($a_5$)] for all $x,y\in S(X)$, there exists $v\in X$ such that $Sv \preceq x,~Sv\preceq y$,
	\item[($a_6$)]$F$ satisfies $f_{1c}$.
\end{description}

\vspace{.3cm} The objective of this paper is to prove common fixed point results for a self-mappings satisfying an implicit function which is general enough to cover several linear as well as some nonlinear contractions. The main results of this paper are based on the following motivations and observations.

\begin{enumerate} [$(i)$]
	\item To provide an example which shows that Theorem \mbox{\ref{berinde-vetro}} is not correct in the present form.
    \item To modify Theorem \mbox{\ref{berinde-vetro}}, we employ the completeness of any subspace $E$ (such that $T(X) \subseteq E\subseteq S(X)$) rather than the completeness of whole space $X$. This point is very vital and also responsible for the failure of Theorem \ref{berinde-vetro}.
	\item To enrich Theorem \mbox{\ref{berinde-vetro}}, we consider a relatively larger class of implicit functions which also cover some nonlinear contraction besides weakening some earlier metrical notions such as: completeness and continuity.
	\item To improve Theorem \ref{berinde-vetro}, the condition $(a_2)$  is replaced by relatively weaker notion of I-regularity.
	\item To prove a fixed point theorem under a relatively weaker condition of Dane\v{s}-type
	(see \mbox{\cite[Definition 1]{danevs1976two}}) which can be viewed as an order-theoretic version of a famous theorem due to \'{C}iri\'{c}\mbox{\cite[Theorem 1]{ciric1974generalization}} for quasi contraction.
	\item To prove a sharper version of Theorem 1 due to Berinde and Vetro \mbox{\cite{berinde2012common}} in the metric setting.
\end{enumerate}

\section{Results on Ordered Metric Spaces}\label{2.2}

 Firstly, we utilize the following example which exhibits that Theorem \mbox{\ref{berinde-vetro}} is not correct in its present form.

\begin{example} Consider $X=\{x_0,x_1,x_2,...,x_n,...\}$ where $x_0=0,x_i=-(\frac{1}{4})^i,i=\{1,2,3,...\}$ with usual metric and usual order. Then, $(X,d,\leq)$ is an ordered metric space. Define two self-mappings $T$ and $S$ on $X$ by: $$T(x_i)=x_{i+2} \text{ and } S(x_i)=x_{i+1}, \text{ for all } i.$$ Consider the function $F\in\mathcal{F}$ defined by \cite[Example 3.1]{berinde2012approximating}:
$$F(t_1,t_2,t_3,t_4,t_5,t_6)=t_1-k t_2, \text{ where } k\in[0,1).$$
With a view to verify  assumption (\ref{rr}) of Theorem \ref{berinde-vetro}, consider $x_i,x_j$ in $X$ with $i<j$ so that

$$\big|Tx_j-Tx_i\big|\leq k\big|Sx_j-Sx_i\big|,$$
$i.e.,$
$$\Big(\frac{1}{4}\Big)^{i+2}-\Big(\frac{1}{4}\Big)^{j+2}\leq k\Big[\Big(\frac{1}{4}\Big)^{i+1}-\Big(\frac{1}{4}\Big)^{j+1}\Big],$$
or
$$\Big(\frac{1}{4}\Big)^{i+2}\Big[1-\Big(\frac{1}{4}\Big)^{j-i}\Big]\leq k\Big(\frac{1}{4}\Big)^{i+1}\Big[1-\Big(\frac{1}{4}\Big)^{j-i}\Big],$$
which means
$$\frac{1}{4}\leq k.$$
\indent Hence, $F$ satisfies $f_{1a}, f_{1c}$ and $f_2$  for $k\in [\frac{1}{4},1)$. Also, all other assumptions of Theorem \ref{berinde-vetro} are satisfied. Observe that the pair  $(T,S)$ has no common fixed point. In fact they do not admit even a coincidence point.

\end{example}

With a view to correct and enrich Theorem \mbox{\ref{berinde-vetro}}, we frame the following definition:

\begin{definition}\label{xg-icu} 
	Let $(X,d,\preceq)$ be an ordered
	metric space and $S$ a self-mapping on $X.$ We say that $(X,d,\preceq)$ is
	\begin{enumerate}
		\item[{$(i)$}]  I-regular if every increasing sequence
		$\{Sx_n\}$ in $X$ converges to $Sx$, admits a subsequence $\{Sx_{n_k}\}$ such that each term of $\{Sx_{n_k}\}$ is comparable with $Sx$ and $Sx\preceq S(Sx)$.
		\item[{$(ii)$}]  D-regular if every decreasing sequence
		$\{Sx_n\}$ in $X$ converges to $Sx$, admits a subsequence $\{Sx_{n_k}\}$ such that each term of $\{Sx_{n_k}\}$ is comparable with $Sx$ and $Sx\succeq S(Sx)$.
		\item[{$(iii)$}]  M-regular if it is both I-regular and D-regular.
	\end{enumerate}
\end{definition}

Now we are equipped to prove our main result as follows:
\begin{theorem}\label{upperE2} Let $(X,d,\preceq)$ be an ordered metric space and $E$ an $\overline{\rm O}$-complete subspace of $X$. Let $(T,S)$ be a pair of self-mappings on $X$ such that $T(X)\subseteq E\subseteq S(X)$ and $T$ is $S$-increasing. Assume that there exists a function $F \in \mathfrak{F}$ satisfying $F_{1a}$ such that, for all $x,y\in X$ (with $Sx\preceq Sy$),
\begin{equation}\label{rr5tr}
F(d(Tx,Ty),d(Sx,Sy),d(Sx,Tx),d(Sy,Ty),d(Sx,Ty),d(Sy,Tx))\leq0.
\end{equation}
If following conditions hold:
\begin{description}
         \item[($b_1$)] there exists $x_{0}\in X$ such that $Sx_0\preceq Tx_0$,
         \item[($b_2$)]$(E,d,\preceq)$ is I-regular,
\end{description}
\noindent then, the pair $(T,S)$ has a coincidence point in $X$. Also, if
\begin{description}
	           \item[($b_3$)] $F$ satisfies $F_2$,
  \item[($b_4$)] $(T,S)$ is weakly compatible pair,
\end{description}
then, the pair  $(T,S)$ has a common fixed point. Moreover, if
        \begin{description}
        \item[($b_5$)] $C(T,S)$ is $(T,S)$-directed,
        \item[($b_6$)] $F$ satisfies $F_{1c}$,
                 \end{description}
\noindent then the common fixed point is unique. Further for any such $x_0$ in $X$, the $T$-$S$-sequence $\{Tx_n\}$ (with initial point $x_0$) converges to the unique common fixed point of the pair.
 \end{theorem}
\begin{proof}
The proof is divided into three steps as follows:

\vspace{.3cm}
{\noindent{\bf Step 1.}}  For $x_0$ with $Sx_0\preceq Tx_0$, we can construct a $T$-$S$-sequence $\{Tx_n\}$ with initial point $x_0$ satisfying
\begin{equation*}\label{maineqt}
  Sx_0\preceq Tx_0=Sx_1\preceq Tx_1= Sx_2\preceq ... =Sx_n\preceq Tx_n=Sx_{n+1}\preceq Tx_{n+1} ... .
\end{equation*}
Clearly, $\{Sx_n\},\{Tx_n\}\subset T(X)\subseteq E$. Moreover, both the sequences are increasing sequences. If $Tx_m=Tx_{m+1}$ for some $m\in \mathbb{N}$, then $x_{m+1}$ is the required coincidence point and we are through. Henceforth, we assume  that $Tx_n\ne Tx_{n+1}$ for all $n\in \mathbb{N}$.
As $Sx_n\preceq Sx_{n+1}$, we can take $x=x_n$ and $y=x_{n+1}$ in (\ref{rr5tr}) so that
\begin{multline*}
F(d(Tx_n,Tx_{n+1}),d(Tx_{n-1},Tx_{n}),d(Tx_{n-1},Tx_{n}),d(Tx_{n},Tx_{n+1}),\\d(Tx_{n-1},Tx_{n+1}),d(Tx_{n},Tx_{n}))\leq0.
\end{multline*}
Since $F$ is decreasing in the fifth variable, on using the triangular inequality, above inequality become
\begin{multline*}F(d(Tx_n,Tx_{n+1}),d(Tx_{n-1},Tx_n),d(Tx_{n-1},Tx_n),d(Tx_n,Tx_{n+1}),\\d(Tx_n,Tx_{n+1})+d(Tx_{n-1},Tx_n),0)\leq0.\end{multline*}
Thus, there exists a comparison function  $\varphi$ such that
\begin{eqnarray}
\label{pphidecr}d(Tx_n,Tx_{n+1})&\leq& \varphi (d(Tx_{n-1},Tx_n)).
\end{eqnarray}
 Since $\varphi $ is increasing function, on using induction on $n$ in (\ref{pphidecr}), we get
\begin{eqnarray*}\label{zmv}
\label{phidecrr}d(Tx_n,Tx_{n+1})&\leq& \varphi^n (d(Tx_0,Tx_1)) \text{, for all } n\in \mathbb{N}_0.
\end{eqnarray*}

\noindent Let $\epsilon$ be fixed. Choose $n\in \mathbb{N}_0$ so that $$d(Tx_{n+1},Tx_n)< \epsilon - \varphi (\varepsilon).$$
Now,
\begin{eqnarray*}
	d(Tx_{n+2},Tx_n)&\leq& d(Tx_{n+2},Tx_{n+1})+ d(Tx_{n+1},Tx_n)\\
	&<&\varphi(d(Tx_{n+1},Tx_n))+ \epsilon - \varphi (\varepsilon)\\
	&\leq&\varphi(\epsilon - \varphi (\varepsilon))+ \epsilon - \varphi (\varepsilon)\\
	&\leq&\varphi(\epsilon)+ \epsilon- \varphi (\varepsilon)=\varepsilon.
\end{eqnarray*}
Also,
\begin{eqnarray*}
	d(Tx_{n+3},Tx_n)&\leq& d(Tx_{n+3},Tx_{n+1})+ d(Tx_{n+1},Tx_n)\\
	&<&\varphi(d(Tx_{n+2},Tx_n)+ \epsilon - \varphi (\varepsilon)\\
	&\leq&\varphi(\epsilon) + \epsilon - \varphi (\varepsilon)=\varepsilon.
\end{eqnarray*}

\vspace{.3cm}
\noindent By induction $$d(Tx_{n+k},Tx_n)<\varepsilon,\text{ for all } k\in \mathbb{N}$$ so that $\{Tx_n\}$ is a Cauchy sequence
 in the  $\overline{\rm O}$-complete subspace $E$. Therefore, there exists some $z \in E \text{ and }x\in X$ such that $z=Sx$ with
 \begin{equation}\label{mmaineqt}
 Tx_n\uparrow Sx\text{ and }Sx_n\uparrow Sx.
 \end{equation}

\vspace{.3cm}
{\noindent{\bf Step 2.}} Since $(E,d,\preceq)$ is I-regular, there exists a subsequence $\{Sx_{n_k}\}$ of $\{Sx_n\}$ such that
\begin{align*}\label{37}
  Sx_{n_k}\prec\succ Sx,~~\forall~k~ \in \mathbb{N}.
\end{align*}
On putting $x=x_{n_k},~y=x$ in (\ref{rr5tr}), one gets
$$F(d(Tx_{n_k},Tx),d(Sx_{n_k},Sx),d(Sx_{n_k},Tx_{n_k}),d(Sx,Tx),d(Sx_{n_k},Tx),d(Sx,Tx_{n_k}))\leq0.$$
As $F$ is a continuous, letting $k\rightarrow \infty$ and using (\ref{mmaineqt}), we get
$$F(d(Sx,Tx),0,0,d(Sx,Tx),d(Sx,Tx),0)\leq0,$$
implying thereby,
$$d(Sx,Tx)\leq \varphi (0)=0$$
so that $Sx=Tx.$

\vspace{.3cm}
{\noindent{\bf Step 3.}}
Since the pair $(T,S)$ is weakly compatible and $Sx=Tx(=z$ for some $z\in X$), we have
\begin{equation}\label{gy=Ty}
  Sz=S(Tx)=T(Sx)=Tz
\end{equation}
By assumption $(b_2)$, $Sx\preceq SSx=Sz$. So, by putting $x=x$ and $y=z$ in (\ref{rr5tr}), we get
$$F(d(Tx,Tz),d(Sx,Sz),d(Sx,Tx),d(Sz,Tz),d(Sx,Tz),d(Sz,Tx))\leq0$$
so that $d(Tx,Tz)=0$ which along with (\ref{gy=Ty}) gives rise $Sz=Tz=Tx=z,$ $i.e.,$ $z$ is a common fixed point of $T$ and $S$.

\vspace{.3cm}
Now, we show that the pair $(T,S)$ has a unique common fixed point in the presence of $(b_5)$ and $(b_6)$. Let $z$ and $y$ be two common fixed points of the pair $(T,S)$. By the $(T,S)$-directedness of $C(T,S)$, there exists some $t_0\in X$ such that $Ty\prec\succ St_0$ and $Tz\prec\succ St_0$. Since $T(X)\subseteq S(X)$ and $T$ is $S$-increasing, we can define a sequence $\{t_n\}\subset
X$ with \begin{equation*} \label{144} St_{n+1}=Tt_n,\end{equation*}
and \begin{equation*} \label{155} Sy\prec\succ St_n
\;\; \forall~ n\in \mathbb{N}_0\end{equation*}
On setting $x=t_n,~y=y$ in (\ref{rr5tr}), 
we have
\begin{eqnarray*}
  0 &\geq& F(d(Tt_{n},Ty),d(St_n,Sy),d(St_n,Tt_n),d(Sy,Ty),d(St_n,Ty),d(Sy,Tt_n)) \\
  &=& F(d(Tt_{n},Ty),d(Tt_{n-1},Ty),d(Tt_{n-1},Tt_n),0,d(Tt_{n-1},Ty),d(Ty,Tt_n)) \\
  &\geq& F\big(d(Tt_{n},Ty),d(Tt_{n-1},Ty),d(Tt_{n-1},Ty)+d(Ty,Tt_n),0,d(Tt_{n-1},Ty),\\
  & &d(Ty,Tt_n)\big)
\end{eqnarray*}
so that
\begin{eqnarray*}
\label{phidecr}d(Tt_{n},Ty)&\leq& \varphi (d(Tt_{n-1},Ty)),
\end{eqnarray*}
for a comparison function $\varphi$.

\vspace{.3cm}
On using argument similar to that in Step 1, we can prove that
$d(Tt_{n},Ty)\leq \varphi^n (d(Tt_0,Ty))$, for all $n\in \mathbb{N},$
which on letting $n\rightarrow \infty$ on both sides, gives rise
 $$d(Tt_{n},Ty)\rightarrow0~as~n\rightarrow\infty.$$


Similarly, we can prove that $$d(Tt_n,Tz)\rightarrow0 ~as~n\rightarrow\infty.$$
 Hence,
$$d(z,y)=d(Tz,Ty)\leq d(Tz,Tt_n)+d(Tt_n,Ty)\rightarrow0,\text{ as } n\rightarrow\infty,$$
which amounts to saying that $z=y$.\\

 From the above proof, it follows that for any $x_0\in X$ satisfying $(b_1)$, the $T$-$S$-sequence $\{Tx_n\}$ (with initial point $x_0$) converges to a unique common fixed point of the pair $(T,S)$.\end{proof}


A comprehension of Theorem \ref{upperE2} and Theorem \ref{berinde-vetro}, We reveals the following facts:
\begin{itemize}
		\item The completeness in Theorem \ref{berinde-vetro} is merely required on any subspace $E$ rather than the whole space $X$ such that $T(X) \subseteq E\subseteq S(X)$. This point is very vital and is also responsible for the failure of the Theorem \ref{berinde-vetro}.
		
		\item The class of the implicit function utilized in Theorem \mbox{\ref{upperE2}}  is relatively larger than the one utilized in Theorem \mbox{\ref{berinde-vetro}}.
		
		
		\item The property embodied in condition $(a_2)$ of Theorem \ref{berinde-vetro} implies the I-regularity (utilized in  Theorem \ref{upperE2}).
		\item The notions on `continuity and completeness' employed in Theorem \mbox{\ref{upperE2}} are relatively weaker than their correspondence notions in Theorem \mbox{\ref{berinde-vetro}}. 	\end{itemize}
	
In fact, we can replace the I-regularity of $(E,d,\preceq)$ together with condition $F_{1c}$  from the hypothesis of Theorem \ref{upperE2} at the cost of the  comparability of one of $T$ and $S$ along with a stronger condition on the set $C(T,S)$, as we shall see in Theorem \ref{upperE2b}.\\

 The following three conditions will be utilized in our forthcoming results:
\begin{description}
  \item[($i$)] $T$ is $(S,{\overline{\rm O}})$-continuous.
  \item[($ii$)] $(T,S)$ is an $\overline{\rm O}$-compatible pair and both $T$ and $S$ are $\overline{\rm O}$-continuous.
  \item[($iii$)] $T$ and $S$ are continuous.\end{description}

\begin{theorem}\label{upperE2b}
Theorem \ref{upperE2} remains true if assumptions ($b_2$), ($b_5$) and ($b_6$) are resp. replaced by the following conditions besides retaining the rest of the hypothesis:
\begin{description}
  \item[($c_2$)] any one of the assumptions ($i$), ($ii$) and ($iii$) is satisfied.
  \item[($c_5$)] $C(T,S)$ is totally ordered set.
  \item[($c_6$)] one of $T$ and $S$ is a comparable mapping.
\end{description}

 \end{theorem}
 \begin{proof}  The proof is divided into three steps where step 1 is the same as in the proof of Theorem \ref{upperE2} and hence omitted. The other two steps are discussed separately as follows:

\vspace{.3cm}
{\noindent{\bf Step 2.}} Using the conditions embedded in assumption ($c_2$), the modified form of Step 2 runs as follows:\\
{\noindent{\bf ($i$):}}
Since $T$ is $(S,\overline{\rm O}$)-continuous and $S(x_n)\uparrow Sx$, we have $Tx_n\uparrow Tx$. Now, owing to the uniqueness of the limit and (\ref{mmaineqt}) we have, $Sx=Tx.$ Thus, we are through.\\

{\noindent{\bf ($ii$):}}
In view of (\ref{mmaineqt}) and the $\overline{\rm O}$-continuity of both $T$ and $S$, we have

\begin{equation*}
\label{gSxn=gz}\lim\limits_{n\to\infty} S(Tx_n)=Sz,
\end{equation*}
and
\begin{equation*}
\label{TSxn=Tz}\lim\limits_{n\to\infty} T(Sx_n)=Tz.
\end{equation*}
Now, the $\overline{\rm O}$-compatibility of the pair $(T,S)$ gives rise, $ Sz=Tz$.\\


{\noindent{\bf ($iii$):}} Our proof runs on the lines of the proof of Theorem 1 in \cite{alam2015enriching}. We reproduce it here for convenience of the readers. Since  $T$ and $S$ are $\overline{\rm O}$-continuous, owing to Lemma \ref{lem2}, there exists a subset
$A\subseteq X$ such that $S(A)=S(X)$ and $S: A \rightarrow X$ is
one-one. Without loss of generality, we can choose
$A$ such that $x\in A.$ Now, define $T: S(A) \rightarrow
S(X)$ by
\begin{align} \label{10}T(Sa)=Ta,\;\;\forall\; Sa\in S(A)\; {\rm where}\; a\in A.\end{align}
As $S: A \rightarrow X$ is one-one and $T(X)\subseteq S(X)$, $T$ is
well defined. Since $\{x_n\}\subset X$ and $S(X)=S(A)$,
there exists $\{a_n\}\subset A$ such that $Sx_n=Sa_n\;\forall~
n\in \mathbb{N}_0.$ By using Lemma \ref{prop1}, we get
$Tx_n=Ta_n,\;\forall~ n\in \mathbb{N}_0.$ Therefore, owing
to (\ref{mmaineqt}), we have
\begin{align} \label{11} Sa_n=Ta_n\uparrow Sx.\end{align}
On using (\ref{10}), (\ref{11}) and continuity of $T$, we get
$$Tx=T(Sx)=T(\lim\limits_{n\to\infty} Sa_n)=\lim\limits_{n\to\infty} T(Sa_n)=\lim\limits_{n\to\infty} Ta_n=Sx,$$
$i.e.,$ $Tx=Sx.$

\vspace{.3cm}
{\noindent{\bf Step 3.}}
As the pair $(T,S)$ is weakly compatible, we have
\begin{equation*}\label{comm}
  Sz=S(Tx)=T(Sx)=Tz
\end{equation*}
By the assumption ($c_5$), $Sx\prec\succ Sz$. Now, on setting $x=x$ and $y=z$ in (\ref{rr5tr}), we get
\begin{eqnarray*}0 &\geq& F(d(Tx,Tz),d(Sx,Sz),d(Sx,Tx),d(Sz,Tz),d(Sx,Tz),d(Sz,Tx)) \\   &=&F(d(Tx,Tz),d(Tx,Tz),d(Tx,Tx),d(Tz,Tz),d(Tx,Tz),d(Tz,Tx)).\end{eqnarray*}
Assumption ($F_2$) implies $d(Tz,Tx)=0$ and hence $Sz=Tz=Tx=z$.\\

Now, we show that the common fixed point $z$ is unique. Let $y,z\in X$ be two common fixed points of the pair $(T,S)$. By repeating earlier arguments, we have $Ty=Tz$. Thus, the pair $(T,S)$ has a unique common fixed point. From the proof it follows that, for any such $x_0\in X$ satisfying $(b_1)$, the $T$-$S$-sequence $\{Tx_n\}$ (with initial point $x_0$) converges to a unique common fixed point.
\end{proof}

\begin{remark}
  When the condition $(c_2)$ is satisfied with $(ii)$ in Theorem \ref{upperE2b}, Remark \ref{rem5} implies that $(T,S)$ is weakly compatible pair. Thus, assumption $(b_4)$ is not required  any more in the hypothesis.
\end{remark}

The following example exhibits that Theorem \ref{upperE2} is genuinely different to Theorem \ref{upperE2b}.
\begin{example}\label{113}
Consider $X=(-1,1]$ equipped with usual metric. Then, $(X,d,\preceq)$
is an ordered metric space wherein for $x,y\in X,$ $$x\preceq y\Leftrightarrow (x\leq y, y\neq1) ~{\rm or}~ (x=y=1).$$ Herein '$\leq$' stand of the usual order on $\mathbb{R}$. Set $S=I_X$ and define $T:X\rightarrow X$ by $$Tx=x/3, ~\rm{for~ all}~ x\in X.$$
Consider $F\in \mathfrak{F}$ given in Example \ref{1.10.of.V.and.V} so that $F$ satisfies $F_{1a}$ and $F_2$ for the comparison function $\psi(t)=kt,$ (for some $k\in(0,1)$). Thus, (by taking $E=X$) Theorem \ref{upperE2b} (with assumption ($ii$)) ensures the existence of a unique common fixed point (namely $x=0$).
Observe that, Theorem \ref{rr5tr} is not applicable not only because $F$ does not satisfy condition $F_{1c}$ but also $(X,d,\preceq)$ is not I-regular. It is worth mentioning here that Theorem \mbox{\ref{berinde-vetro}} is not applicable to present example due to the involvement of relatively weaker completeness notion.


\end{example}
Though, the succeeding two theorems are similar (to Theorems \ref{upperE2} and \ref{upperE2b}), yet there do exist instances wherein the following two theorems are applicable but Theorems \ref{rr5tr} and \ref{upperE2b} are not, which substantiate the utility of such theorems.


\begin{theorem}\label{upperE1}
	Let $X,~E,~T$ and $S$ be defined as in Theorems \ref{upperE2}. Assume there exists a function $F \in \mathfrak{F}$ satisfying conditions $F_{1a}$ and (\ref{rr5tr}).
Suppose that the following conditions hold:
\begin{description}
         \item[($d_1$)] there exists $x_{0}\in X$ such that $Sx_0\preceq Tx_0$,
         \item[($d_2$)]$(E,d,\preceq)$ is I-regular,
                 \end{description}
\noindent then the pair $(T,S)$ has a coincidence point in $X$. Also, if
\begin{description}
        \item[($d_3$)] $(C(T,S),\preceq)$ is $(T,S)$-directed,
        \item[($d_4$)] $F$ satisfies $F_{1b}$,
                 \end{description}
then, the pair $(T,S)$  has a unique point of coincidence. And if
\begin{description}
  \item[($d_5$)] one of $T$ and $S$ is one-one,
\end{description}
  then, the pair $(T,S)$ has a unique coincidence point. Moreover if
        \begin{description}
        \item[($d_6$)] $(T,s)$ is weakly compatible pair,
                 \end{description}
\noindent then, the pair  $(T,S)$ has a unique common fixed point. Further, for any $x_0\in X$, the $T$-$S$-sequence $\{Tx_n\}$ (with initial point $x_0$) converges to the unique common fixed point.
 \end{theorem}

\begin{proof} The proof is divided into five steps where Step 1 and Step 2 are the same as in the proof of Theorem \ref{upperE2} and hence omitted. Other steps run as follows:

\vspace{.3cm}
{\noindent{\bf Step 3.}} Let $x,y,\overline{x},\overline{y}\in X$ be such that
 \begin{align} \label{13t} Sx=Tx=\overline{x}~\text{ and }Sy=Ty=\overline{y}. \end{align}


\noindent We assert that $\overline{x}=\overline{y}$. Due to the $(T,S)$-directedness of $X$, there exists $t_0\in X$ such that $Sx\prec\succ St_0 \text{ and } Sy\prec\succ St_0.$ Now, for $Sx\prec\succ St_0$, we can define a sequence $\{t_n\}\subset
X$ with \begin{align} \label{14t} St_{n+1}=Tt_n,\end{align}
and \begin{equation*} \label{15} Sx\prec\succ St_n~, \forall~ n\in \mathbb{N}_0.\end{equation*}

We claim that
\begin{equation*} \label{16} \lim\limits_{n\to\infty}d(Tx,Tt_n)=0.\end{equation*}
Two cases arise:

\vspace{.3cm}\noindent
Firstly, if $d(Tx,Tt_m)=0$ for some $m\in \mathbb{N}_0$,
then by (\ref{13t}) and (\ref{14t}), we get $$d(Sx,St_{m+1})=0.$$
Consequently, by Lemma \ref{prop1}, we must have, $d(Tx,Tt_{m+1})=0$. By induction on $m$, we get $d(Tx,Tt_n)=0$, for all $n>m$.

\vspace{.3cm}
\noindent Secondly, suppose that $d(Tx,Tt_n)>0,$ for all $n\in \mathbb{N}_0$. On putting $x=x,~y=t_n$ in (\ref{rr5tr}) and using assumption ($d_4$), we have
\begin{eqnarray*} \label{its}
0&\geq& F(d(Tx,Tt_n),d(Sx,St_n),d(Sx,Tx),d(St_n,Tt_n),d(Sx,Tt_n),d(St_n,Tx)\\
&=& F(d(Tx,Tt_n),d(Tx,Tt_{n-1}),0,d(Tt_{n-1},Tt_n),d(Tx,Tt_n),d(Tt_{n-1},Tx))\\
&\geq& F(d(Tx,Tt_n),d(Tx,Tt_{n-1}),0,d(Tx,Tt_n)+d(Tx,Tt_{n-1}),d(Tx,Tt_n),\\
& & d(Tt_{n-1},Tx))
\notag\end{eqnarray*}
so that there exists a comparison function $\varphi$ with
\begin{eqnarray}
\label{zdc}d(Tx,Tt_n)&\leq& \varphi (d(Tx,Tt_{n-1})).
\end{eqnarray}
Since $\varphi $ is increasing function, owing to the induction on $n$ [in (\ref{zdc})], we have
\begin{eqnarray*}
\label{zdcx}d(Tx,Tt_n)&\leq& \varphi^n (d(Tx_,Tx_0)), \text{ for all } n\in \mathbb{N}_0.
\end{eqnarray*}
Letting $n\rightarrow \infty$ on the both sides, 
we find $d(Tx,Tt_n)\rightarrow0.$ Thus, in all, the claim is established.

\vspace{.3cm}
Similarly, for $Sy\prec\succ St_0$, one can show that $$\lim\limits_{n\to\infty}d(Ty,Tt_n)=0.$$
Now,
\begin{eqnarray*}
d(\overline{x},\overline{y})&=&d(Tx,Ty)\\
&\leq& d(Tx,Tt_{n})+d(Tt_{n},Ty) \rightarrow~0~as~n\rightarrow\infty
\end{eqnarray*}
Thus, the pair $(T,S)$ has a unique point of coincidence.

\vspace{.3cm}
{\noindent{\bf Step 4.}} Let $T$ be one-one. On contrary, assume that there exist two coincidence points $x,y\in X$ such that$$Sx=Tx=Sy=Ty.$$ As $T$ is one-one, we have $x=y$. The  similar arguments carries over in case $S$ is one-one.

\vspace{.3cm}
{\noindent{\bf Step 5.}} Let $x,\overline{x}\in X$ be such that 
$Sx=Tx=\overline{x}.$ By Lemma \ref{prop2}, $\overline{x}$ itself is a coincidence point. In view of step 4, we must have $\overline{x}=x$ and hence we are through.


\vspace{.3cm}\indent From the above proof it follows that, for any $x_0\in X$ satisfying $(d_1)$, the $T$-$S$-sequence $\{Tx_n\}$ (with initial point $x_0$) converges to a unique common fixed point.\end{proof}

\begin{theorem}\label{upperE12}
Theorem \ref{upperE1} remains true if the
condition ($d_2$) is replaced by any one of the conditions ($i$), ($ii$) and ($iii$) (besides retaining the rest of the hypotheses).
\end{theorem}
\begin{proof}
In the proof of the theorem, Steps 1 and 2 are the same as in the proof of Theorem \ref{upperE2b} while steps 3, 4 and 5 are the same as in the proof of Theorem \ref{upperE1}.
\end{proof}

\begin{remark}\label{duals}
  One can obtain dual type results corresponding to all preceding theorems by replacing ``$\overline{\rm O}$-analogues'' with ``$\underline{\rm O}$-analogues'', the ``I-regularity'' with  ``D-regularity'' and the condition ``$Sx_0\preceq Tx_0$'' with ``$Sx_0\succeq Tx_0$''.
 \end{remark}
\begin{remark}\label{compainedco}
One can obtain a companied type result corresponding to all preceding theorems by replacing ``$\overline{\rm O}$-analogues'' with ``${\rm O}$-analogues'', the ``I-regularity'' with  ``M-regularity'' and 	`$Sx_0\preceq Tx_0$'' with $Sx_0\prec\succ Tx_0$''.

 \end{remark}

All the above results unify, extend and generalize many relevant common fixed point results from the existing literature which can not be completely mentioned here. But a sample, we  consider a famous theorem due to \'{C}iri\'{c} \cite[Theorem 1]{ciric1974generalization} and extend the same to a pair of self-mappings satisfying Dane\v{s}-type contraction in an ordered metric space.



\begin{corollary}\label{qquasi} Let $X,~E,~T$ and $S$ be defined as in Theorems \ref{upperE2}. Suppose there exists a continuous half-comparison function $\rho$ such that for all $Sx\preceq Sy$,
	\begin{equation*}\label{etr}  d(Tx,Ty)\leq \rho (max\{d(Sx,Sy),d(Sx,Tx),d(Sy,Ty), d(Sx,Ty), d(Sy,Tx)\}).
	\end{equation*}
	If the following conditions hold:
	\begin{description}         		\item[$(g_1)$] there exists $x_{0}\in X$ such that $Sx_0\preceq Tx_0$,         		\item[$(g_2)$] $(E,d,\preceq)$ is I-regular (resp. any one of the assumptions ($i$), ($ii$) and ($iii$) is satisfied),
	\end{description}
	then, the pair  $(T,S)$ has a coincidence point in $X$. Also, if
	\begin{description}               		
		\item[$(g_3)$] $(T,s)$ is weakly compatible pair,
		\end{description}
		\begin{description}        		
			\item[$(g_4)$] $(C(T,S),\preceq)$ is $(T,S)$-directed (resp. $C(T,S)$ is totally ordered set and one of $T$ and $S$ is comparable),
		\end{description}
	 then, the pair  $(T,S)$ has a unique common fixed point. Further for any such $x_0\in X$, the $T$-$S$-sequence $\{Tx_n\}$ with initial point $x_0$ converges to the unique common fixed point.
	 \end{corollary}
\begin{proof}	
The result is obtained from Theorem \ref{upperE2} (resp. Theorem \ref{upperE2b}) by taking the function $F\in \mathfrak{F}$ defined by Example \ref{nonlin.quasi}.
\end{proof}
 Notice that, above nonlinear-result ($i.e.,$ Corollary \mbox{\ref{qquasi}}) can not be derived using Theorem \ref{berinde-vetro}, $i.e.,$ \cite[Theorem 2]{berinde2012common}.
\begin{remark}
	Similarly, one can obtain the dual of Corollary \ref{qquasi} corresponding to Theorems \ref{upperE1} and \ref{upperE12}.
\end{remark}
Observe that, setting $\rho(t)=kt,~k\in[0,\frac{1}{2})$ in Corollary \ref{qquasi}, gives rise a linear form of the corollary. Interestingly, we show that this linear form is not valid for $k\geq\frac{1}{2}$. Consequently, Corollary \ref{qquasi} is not valid for a general comparison function.

\begin{example} Consider $X=[0,\infty)$ with usual metric and usual order. Then, $(X,d,\leq)$	is an ordered metric space. Define $T,S:X\to X$	by $$Tx=x^2+1\;{\rm and}\;Sx=\frac{2}{3}x, \;\forall~ x\in X.$$
		By a routine calculation, one can verify that all the conditions of Corollary \ref{qquasi} are satisfied with $k\geq\frac{1}{2}$. 
	 Nevertheless, the pair $(T,S)$ has no coincidence in $X$. For, if $x$ is a coincidence point, we must have a real root for $3x^2-2x+3=0$ which is not true.
	\end{example}

\section{Corresponding Results on Metric Spaces}\label{subsec2.3}

We can deduce the following sharpened version of Theorem 1 due to Berinde and Vetro \cite{berinde2012common}.
\begin{theorem}\label{mupperE2}
Let $(X,d)$ be a metric space and $E$ a complete subspace of $X$. Let $(T,S)$ be a pair of self-mappings on $X$ such that $T(X)\subseteq E\subseteq S(X)$. Assume that there exists a function $F \in \mathfrak{F}$ satisfying $F_{1a},$ such that for all $x,y\in X$
\begin{equation*}\label{rrr7u}
F(d(Tx,Ty),d(Sx,Sy),d(Sx,Tx),d(Sy,Ty),d(Sx,Ty),d(Sy,Tx))\leq0.
\end{equation*}
\noindent Then the pair  $(T,S)$ has a coincidence point in $X$. Moreover, if $(T,s)$ is weakly compatible pair and $F$ satisfies $F_{2}$, then the pair  $(T,S)$ has a unique common fixed point. Further for any $x_0\in X$, the $T$-$S$-sequence $\{Tx_n\}$ (with initial point $x_0$) converges to the unique common fixed point of the pair $(T,S)$.
 \end{theorem}
\begin{proof}The proof is omitted as it is similar to the one given in \cite[Theorem 1]{berinde2012common}, except some minor changes  corresponding to the new implicit function.\end{proof}

\begin{corollary}\label{quasiinmetric}
Let $X,~E,~T$ and $S$ be defined as in Theorems \ref{mupperE2}. Assume that there exists a continuous half-comparison function $\rho$ such that,
 for all $x,y\in X$
 \begin{equation*}\label{eturrr}  d(Tx,Ty)\leq \rho (max\{d(Sx,Sy),d(Sx,Tx),d(Sy,Ty), d(Sx,Ty), d(Sy,Tx)\}).
 \end{equation*}
Then, the pair  $(T,S)$ has a unique common fixed point. Further for any such $x_0\in X$, the $T$-$S$-sequence $\{Tx_n\}$ (with initial point $x_0$) converges to the unique common fixed point of the pair $(T,S)$.
\end{corollary}
\begin{remark}
	Setting $\rho(t)=kt$ (where $k\in[0,1/2)])$, $S=I_X$ and $E=T(X)$ in Corollary \ref{quasiinmetric}, we reduces it to a partially sharpened version of Theorem 1 due to \'{C}iri\'{c} \cite{ciric1974generalization}.
\end{remark}
\begin{remark}
	One can drive dual type results corresponding to the results of this section as indicated in Remarks \ref{duals} and \ref{compainedco}.
\end{remark}

\vspace{.2cm}
\noindent{\bf Author's Contributions:} All authors read and approved
the final manuscript.


\vspace{.2cm}	
\noindent{\bf Conflict of Interests:} All the authors declare that there is no conflict of interests regarding the publication of this paper.


\begin{thebibliography}{99}
\bibitem{popa1997fixed} Valeriu Popa. {\em Fixed point theorems for implicit contractive mappings,} 
stud. Cerc. St. Ser. Mat. Univ. Bacau, {\bf 7}, 127-133, 1997.

\bibitem{popa1999some} Valeriu Popa. {\em Some fixed point theorems for compatible mappings
	 satisfying an implicit relation,} Demonstratio Mathematica, {\bf 32}(1), 157-163, 1999.
 
\bibitem{ali2009unifying} Javid Ali and Mohammad Imdad. {\em Unifying a multitude of common fixed point theorems
	 employing an implicit relation,} Commun. Korean Math. Soc, {\bf 24}(1), 41-55, 2009.
 
\bibitem{altun2009some} Ishak Altun and Duran Turkoglu. {\em Some fixed point theorems for weakly compatible mappings 
satisfying an implicit relation,} Taiwanese Jour. Math, {\bf 13}(4), 2009.

\bibitem{aliouche2007common} Abdelkrim Aliouche and Ahcene Djoudi. {\em Common fixed point theorems for mappings satisfying 
	an implicit relation without decreasing assumption,} Hacett. Jour. Math. Stat., {\bf 36}(1), 2007.

\bibitem{popa2005general} Valeriu Popa. {\em A general fixed point theorem for four weakly compatible mappings 
	satisfying an implicit relation,} Filomat, {\bf 19}, 45-51, 2005.

\bibitem{aliouche2009general} Abdelkrim Aliouche and Valeriu Popa. {\em General common fixed point theorems for occasionally
	weakly compatible hybrid mappings and applications,} Novi Sad J. Math., {\bf 39}(1), 89-109, 2009.

\bibitem{berinde2012approximating} Vasile Berinde. {\em Approximating fixed points of implicit almost contractions,} Hacett.
	Jour. Math. Stat., {\bf 41}(1), 2012.
	
\bibitem{popa2001general} Valeriu Popa. {\em A general fixed point theorem for weakly compatible mappings in compact metric
	spaces,} Turkish Jour. Math., {\bf 25}(4), 465-474, 2001.

\bibitem{popa2002fixed} Valeriu Popa. {\em Fixed points for non-surjective expansion mappings satisfying an implicit relation,}
		Bul. Stiint. Univ. Baia Mare Ser. B Fasc. Mat.-Inform, {\bf 18}, 105-108, 2002.
		
\bibitem{vetro2013common} Calogero Vetro and Francesca Vetro. {\em Common fixed points of mappings satisfying implicit relations 
	in partial metric spaces,} Jour. Nonlinear Sci. Appl, {\bf 6}(3), 152-161, 2013.

\bibitem{ullah2016unified} Md Ahmadullah, Javid Ali and Mohammad Imdad. {\em Unified relation-theoretic metrical fixed point
	theorems under an implicit contractive condition with an application,} Fixed Point Theory Appl., {\bf 2016:42}, 1-15, 2016.

\bibitem{popa2010using} Valeriu Popa, Mohammad Imdad and Javid Ali. {\em Using implicit relations to prove unified fixed
	point theorems in metric and 2-metric spaces,} Bull. Malays. Math. Sci. Soc.(2), {\bf 33}(1), 105-120, 2010.

\bibitem{imdad2008general} Mohammad Imdad and Javid Ali. {\em A general fixed point theorem in fuzzy metric spaces via an
	implicit function,} Jour. Appl. Math. Infor., {\bf 26}(3-4), 591-603, 2008.

\bibitem{ali2008implicit} Javid Ali and Mohammad Imdad. {\em An implicit function implies several contraction conditions,} 
Sarajevo J. Math, {\bf 4}(17), 269-285, 2008.

\bibitem{imdad2002remarks} Mohammad Imdad, Santosh Kumar and Mohammad S. Khan. {\em Remarks on some fixed point
	theorems satisfying implicit relations,} Rad. Mat, {\bf 11}(1), 135-143, 2002.

\bibitem{imdad2013employing} Mohammad Imdad, Mohd Hasan, Hemant Kumar Nashine and Penumarthy P. Murthy. {\em Employing
	an implicit function to prove unified common fixed point theorems for expansive type mappings in 
	symmetric spaces,} J. Nonlinear Anal. Appl, {\bf 2013}, 2013.

\bibitem{berinde2012common} Vasile Berinde and Francesca Vetro. {\em Common fixed points of mappings satisfying implicit contractive
	conditions,} Fixed Point Theory Appl., {\bf 2012:105}, 1-8, 2012.
\bibitem{o2008fixed} Donal O'Regan and Adrian Petru\c{s}el. {\em Fixed point theorems for generalized contractions in ordered metric spaces,}
 Jour. Math. Anal. Appl., {\bf 341}(2), 1241-1252, 2008.
\bibitem{ciric2009monotone} Ljubomir B. C\'{i}r\'{i}c, Nenad Cak\'{i}c, Miloje Rajov\'{i}c, and Jeong Sheok Ume. {\em Monotone generalized 
	nonlinear contractions in partially ordered metric spaces,} Fixed Point Theory Appl., {\bf 2008:131294}, 1-11, 2008.
\bibitem{jungck1976commuting} Gerald Jungck. {\em Commuting mappings and fixed points,} Amer. Math. Month., {\bf 83}(4), 261-263, 1976.
\bibitem{jungck1986compatible} Gerald Jungck. {\em Compatible mappings and common fixed points,} Inter. Jour. Math. Math. Sci., {\bf 9}(4), 771-779, 1986.
\bibitem{jungck1996common} Gerald Jungck. {\em Common fixed points for noncontinuous nonself maps on nonmetric spaces,}
 Far East. Jour. Math. Sci., {\bf 4}, 199-215, 1996.
\bibitem{alam2015enriching} Aftab Alam, Qamrul Haq Khan and Mohammad Imdad. {\em Enriching some recent coincidence theorems for nonlinear 
	contractions in ordered metric spaces,} Fixed Point Theory Appl., {\bf 2015:141}, 1-14, 2015.
\bibitem{alam2015comparable} Aftab Alam and Mohammad Imdad. {\em Comparable linear contractions in ordered metric spaces,} arXiv:1507.08987, 2015.
\bibitem{sastry2000common} KPR Sastry and ISR Krishna Murthy. {\em Common fixed points of two partially commuting tangential
	selfmaps on a metric space,} Jour. Math. Anal. Appl., {\bf 250}(2), 731-734, 2000.
\bibitem{alam2014some} Aftab Alam, Abdur Rauf Khan and Mohammad Imdad. {\em Some coincidence theorems for generalized nonlinear
	 contractions in ordered metric spaces with applications,} Fixed Point Theory Appl., {\bf 2014:216}, 1-30, 2014.
\bibitem{haghi2011some} RH Haghi, Sh Rezapour and N Shahzad. {\em Some fixed point generalizations are not real generalizations,}
 Nonlinear Analysis: Theory, Methods Appl., {\bf 74}(5), 1799-1803, 2011.
\bibitem{matkowski1975integrable} Janusz Matkowski. {\em Integrable solutions of functional equations,} Proc. Amer. Math. Soc., 1975.
\bibitem{danevs1976two} Josef Dane\v{s}. {\em Two fixed point theorems in topological and metric spaces,} Bull. Austr. Math. Soc., {\bf 14}(02), 259-265, 1976.
\bibitem{ciric1974generalization} Ljubomir B. C\'{i}r\'{i}c. {\em A generalization of Banach's contraction principle,} Proc. Amer. Math. Soc., {\bf 45}(2), 267-273, 1974.
\end{thebibliography}
\end{document}